\documentclass[12pt, notitlepage]{amsart}
\usepackage{latexsym, amsfonts, amsmath, amssymb, amsthm, cite}

\newcommand{\C} {{\mathbb C}}                              
\newcommand{\R} {{\mathbb R}}                              
\newcommand{\Z} {{\mathbb Z}}                              

\newcommand{\F} {{\mathbb F}}

\newcommand {\bv} {\boldsymbol}

\newtheorem{thm}{Theorem}[section]
\newtheorem{prop}[thm]{Proposition}

\newtheorem{lem}[thm]{Lemma}

\usepackage{graphicx}
\usepackage{mathrsfs}

\begin{document}

\title[On Character Sums and Exponential Sums over GAPs]{On Character Sums and Exponential Sums over Generalized Arithmetic Progressions}
\author{Xuancheng Shao}
\address{Department of Mathematics \\ Stanford University \\
450 Serra Mall, Bldg. 380\\ Stanford, CA 94305-2125}
\email{xshao@math.stanford.edu}

\maketitle

\begin{abstract}
Let $\chi$ (mod $q$) be a primitive Dirichlet character. In this paper, we prove a uniform upper bound of the character sum $\sum_{a\in A}\chi(a)$ over all proper generalized arithmetic progressions $A\subset\Z_q$ of rank $r$:
\[ \sum_{n\in A}\chi(n)\ll_{r}q^{\frac{1}{2}}(\log q)^r. \]
This generalizes the classical result by P\'{o}lya and Vinogradov. Our method also applies to give a uniform upper bound for the polynomial exponential sum $\sum_{n\in A}e_q(h(n))$ ($q$ prime), where $h(x)\in\Z[x]$ is a polynomial of degree $2\leq d<q$.
\end{abstract}

\section{Introduction}

Let $q$ be a positive integer and write $\Z_q=\Z/q\Z$. In this paper we study upper bounds on the quantity
\[ S(f,r)=\max_{A}\left|\sum_{n\in A}f(n)\right|, \]
where $f:\Z_q\rightarrow\C$ is an arbitrary function, and the maximum is taken over all  proper generalized arithmetic progressions (GAP) of rank $r$ in $\Z_q$. Recall that a GAP of rank $r$ in $\Z_q$ is of the form
\[ A=\{a_0+a_1h_1+\ldots+a_rh_r\pmod q:0\leq h_1<H_1,\ldots,0\leq h_r<H_r\}, \]
for some $a_0,a_1,\ldots,a_r\in\Z_q$ and positive integers $2\leq H_1,\ldots,H_r\leq q$. We say that $A$ is proper if $|A|=H_1H_2\cdots H_r$. 

Our approach to bound $S(f,r)$ is via the classical completion method using Fourier analysis. We study two different types of functions $f$: $f=\chi$, a primitive Dirichlet character modulo $q$, and $f(n)=e_q(h(n))$ for some polynomial $h(x)\in\Z[x]$ (where, as usual, $e_q(x)=\exp(2\pi ix/q)$).

\subsection{Character sum over GAPs}

We first consider $f=\chi$, a primitive Dirichlet character modulo $q$. In the simplest case when $r=1$, $S(\chi,1)$ is the maximum of character sums over intervals. The first upper bound for $S(\chi,1)$, discovered independently by P\'{o}lya and Vinogradov in 1918, asserts that
\begin{equation}\label{eq:pv} S(\chi,1)\ll q^{\frac{1}{2}}\log q, \end{equation}
uniformly over all primitive Dirichlet characters $\chi$ (mod $q$). This bound has been recently improved by Granville and Soundararajan \cite{granville2007}, and later by Goldmakher \cite{goldmakher2009}, for characters of (small) fixed odd order $g$. They improved the exponent of $\log q$ from $1$ to $1-\delta(g)$ for some $\delta(g)>0$. 

The main theorem of this paper is a generalization of the Polya-Vinogradov type bound for $r>1$:

\begin{thm}\label{main_thm}
Let $\chi$ (mod $q$) be a primitive Dirichlet character. Then for any $r\geq 1$,
\begin{equation}\label{main_eq_epsilon} S(\chi,r)\ll q^{\frac{1}{2}}(\log q)^r, \end{equation}
where the implied constant depends only on $r$.
\end{thm} 

The properness of $A$ in $\Z_q$ is essential here. In fact, Theorem \ref{main_thm} is false if the assumption that $A\subset\Z_q$ is a proper GAP is lifted to $\Z$. For example, if $\chi$ is even and $A=\{h_1+h_2(q-1):1\leq h_1,h_2\leq H\}$ for $H=q/2$, then $A$ is a proper GAP in $\Z$, and
\[ \sum_{a\in A}\chi(a)=\sum_{h_1,h_2=1}^H\chi(h_1-h_2)=2\sum_{1\leq n\leq H}\chi(n)(H-n). \]
By a standard averaging argument, one can show that this sum above can be made as large as $H\sqrt{q}$ by appropriately choosing $\chi$. The extra $H$ factor represents the quantity $\max_{b\in\Z_q}\#\{a\in A:a\equiv b\pmod q\}$, which is $1$ if $A$ is assumed to be proper in $\Z_q$. Our results can be easily generalized to take this into account; however, we will simply focus on the case when $A$ is proper in $\Z_q$.

The proof of Theorem \ref{main_thm} is based on Fourier analysis. As is well known by now, the proof of (\ref{eq:pv}) is based on bounding the $\ell^1$ norm of the Fourier coefficients of an interval $A$ (see, for example, Chapter 23 in \cite{davenport2000}). We apply the same method to treat Theorem \ref{main_thm}. As a result, we will need a bound for the $\ell^1$ norm of the Fourier coefficients of a generalized arithmetic progression, which will be established in Section \ref{sec:l1norm}. 

For comparison, we state here a recent result by Chang \cite{chang2008}, which gives a Burgess type bound for character sums over short GAPs.

\begin{thm}[(Chang)]\label{chang}
Let $p$ be prime. Let $\chi$ (mod $p$) be a nontrivial Dirichlet character. Let $A\subset\F_p$ be a proper GAP of rank $r$ with
\[ |A|>p^{\frac{2}{5}+\epsilon} \]
for some $\epsilon>0$. Then we have
\[ \left|\sum_{a\in A}\chi(a)\right|<p^{-\tau}|A|, \]
where $\tau=\tau(\epsilon,r)>0$ and assuming $p>p(\epsilon,r)$.
\end{thm}

\subsection{Polynomial Exponential sum over GAPs}

In this subsection, we shall restrict our attention to prime $q$. Our method also allows us to establish a bound for polynomial exponential sums over a proper GAP.

\begin{thm}\label{thm:exp}
Let $q$ be prime. Let $h(x)=c_dx^d+\ldots+c_1x+c_0\in\Z[x]$ be a polynomial of degree $2\leq d<q$ with $q\nmid c_d$. Let $f:\Z_q\rightarrow\C$ be defined by $f(n)=e_q(h(n))$. Let $A\subset\Z_q$ be a proper GAP of rank $r$. Then,
\[ S(f,r)\ll dq^{\frac{1}{2}}(\log q)^r, \]
where the implied constant depends only on $r$.
\end{thm}

The quantity $S(f,1)$ has been studied intensively. The classical Weil's bound for complete sums says that
\[ S(f,1)\leq (d-1)q^{\frac{1}{2}}\log q. \]
In fact, our proof of Theorem \ref{thm:exp} will use Weil's bound. It's natural to ask if one can get a nontrivial bound for $d>\sqrt{q}$. This is possible for certain (sparse) polynomials $h$, but we will not consider this case here. For more details on complete polynomial exponential sums, see \cite{bourgain2010exp} and the references therein. 

\subsection{Multilinear character sum over GAPs}

We also prove the following generalization to multilinear character sums. 

\begin{thm}\label{main_gen_thm}
Let $\chi$ (mod $q$) be a primitive Dirichlet character. Let $\bv a_0,\bv a_1,\cdots,\bv a_r\in\Z_q^s$, and let $A\subset\Z_q^s$ be the set
\[ A=\{\bv a_0+h_1\bv a_1+\cdots+h_r\bv a_r\pmod q:0\leq h_i<H_i\}. \]
Suppose that $|A|=H_1H_2\cdots H_r$. Then,
\begin{equation}\label{main_gen_eq_epsilon} \sum_{\bv a=(a_1,\ldots,a_s)\in A}\chi(a_1a_2\cdots a_s)\ll q^{\frac{s}{2}}(\log q)^r, \end{equation}
where the implied constant depends only on $r,s$.
\end{thm}

To treat Theorem \ref{main_gen_thm}, it is necessary to bound the $\ell^1$ norm of Fourier coefficients of GAPs in $\Z_q^s$, which will be done in Section \ref{sec:l1norm}. For comparison, we state the following result by Bourgain and Chang \cite{bourgain2010}, which is a Burgess type bound for multilinear character sums.

\begin{thm}[(Bourgain and Chang)]\label{bourgain_chang}
Let $p$ be prime. Let $\chi$ (mod $p$) be a nontrivial Dirichlet character. Let $L_1,L_2,\ldots,L_n$ be linearly independent linear forms in $n$ variables $x_1,x_2,\ldots,x_n$. Let $I_i=[N_i+1,N_i+H]$ ($1\leq i\leq n$) be intervals. Let
\[ A=\{(L_1(x_1,\ldots,x_n),\ldots,L_n(x_1,\ldots,x_n)):x_i\in I_i\} \]
be a subset of $\F_p^n$. Assume that
\[ H>p^{\frac{1}{4}+\epsilon}. \]
Then,
\[ \sum_{(a_1,\ldots,a_n)\in A}\chi(a_1a_2\cdots a_n)\ll p^{-\delta}|A| \]
for some $\delta=\delta(n,\epsilon)>0$ and $p$ large enough.
\end{thm}

Note that the constraint $H>p^{1/4+\epsilon}$ in Theorem \ref{bourgain_chang} is weaker than the constraint $|A|>p^{2/5+\epsilon}$ in Theorem \ref{chang}. This is due to the use of results from the geometry of numbers in Theorem \ref{bourgain_chang} (first observed by Konyagin \cite{konyagin2010}), which leads to better bounds than from purely additive combinatorial arguments. 

\vspace{5 mm}
The structure of the paper is as follows. In Section \ref{sec:cong}, we establish a lemma needed on bounding the number of solutions to a system of linear congruences, which might be of indpendent interest. In Section \ref{sec:l1norm}, we use this lemma to bound the $\ell^1$ norm of Fourier coefficients of generalized arithmetic progressions. Finally, in Section \ref{sec:gen_fr}, we complete the proofs of Theorem \ref{main_thm}, \ref{thm:exp}, and \ref{main_gen_thm}, using Fourier analysis. The method used in this section is standard, and we include it for the sake of completeness. In Section \ref{sec:conj}, we pose some related open questions.

\vspace{5 mm}
\textit{Acknowledgement: } The author would like to thank his advisor Kannan Soundararajan for carefully reading the first draft of this paper and making many valuable suggestions. He is also grateful for the anonymous referee for proposing the smoothing technique used in the proof of Proposition \ref{prop:num_solns_gen}, along with numerous other helpful comments.


\section{Solutions to Linear Congruences in Short Intervals}\label{sec:cong}

In this section, we establish some lemmas that will be used later. The following proposition, which may be of independent interest, gives an upper bound for the number of solutions to a system of linear congruence equations in short intervals. In the special case when $r=1$, this is to be compared with Lemma 5 in \cite{bourgain2012}.

\begin{prop}\label{prop:num_solns}
Let $q$ be a positive integer. Let $I_1,\ldots,I_r$ be intervals, and $H_1,\ldots,H_r\leq q$ be positive integers. Given nonzero elements $a_1,\ldots,a_r\in\Z_q$ such that the equation
\[ a_1z_1+\cdots+a_rz_r\equiv 0\pmod q \]
has no nontrivial solutions $z_i\in\Z$ with $|z_i|<H_i$ ($1\leq i\leq r$). Let $N$ be the number of solutions of
\[ x_i\equiv a_iy\pmod q,\quad 1\leq i\leq r, \]
in $x_i\in I_i$ and $y\in\Z_q$. Then
\[ N\ll\frac{1}{q^{r-1}}\prod_{i=1}^r\left(|I_i|+\frac{q}{H_i}\right), \]
where the implied constant depends only on $r$.
\end{prop}

Instead of proving Proposition \ref{prop:num_solns}, we state and prove a more general version, which will be used to deal with multilinear character sums.

\begin{prop}\label{prop:num_solns_gen}
Let $q$ be a positive integer. Let $I_1,\ldots,I_r$ be intervals, and $H_1,\ldots,H_r\leq q$ be positive integers. Given nonzero vectors $\bv a_1,\ldots,\bv a_r\in\Z_q^s$ such that the equation
\[ z_1\bv a_1+\ldots+z_r\bv a_r\equiv 0 \pmod q \]
has no nontrivial solutions $z_i\in\Z$ with $|z_i|<H_i$ ($1\leq i\leq r$). Let $N$ be the number of solutions of
\[ x_i\equiv \bv a_i\bv \cdot \bv y\quad \pmod q,\quad 1\leq i\leq r, \]
in $x_i\in I_i$ and $\bv y\in\Z_q^s$. Then,
\[ N\ll q^{s-r}\prod_{i=1}^r\left(|I_i|+\frac{q}{H_i}\right). \]
\end{prop}

Our argument will use repeatedly the following simple estimate on exponential sums over an interval.

\begin{lem}\label{lem:simple_exp_sum}
Let $q,H$ be positive integers and $c\in\Z$. Then,
\[ \sum_{h=0}^{H-1}e_q(ch)\ll\min\left\{H,\left\|\frac{c}{q}\right\|^{-1}\right\}, \]
where $\|x\|$ denotes the distance from $x$ to its nearest integer.
\end{lem}

\begin{proof}[Proof of Proposition \ref{prop:num_solns_gen}]
We write $N$ as
\[ N=\sum_{\substack{x_i\in I_i\\ 1\leq i\leq r}}\sum_{\substack{\bv y\in\Z_q^s \\ \bv{y\cdot a_i}\equiv x_i}}1. \]
Write $L_i=|I_i|$ and let $I_i=[w_i-L_i/2,w_i+L_i/2]$. Note that for any fixed $1\leq i\leq r$ and $x_i\in I_i$,
\[ \sum_{\substack{0\leq\ell_i,\ell_i'\leq L_i\\ x_i=w_i+\ell_i-\ell_i'}}1\gg L_i. \]
Hence,
\begin{align*}
N &\ll\sum_{\substack{x_i\in\Z\\ 1\le i\le r}} \sum_{\substack{\bv y\in\Z_q^s \\ \bv{y\cdot a_i}\equiv x_i}} \sum_{\substack{ 0\le\ell_i,\ell_i'\le L_i \\ x_i=w_i+\ell_i-\ell_i'}} \frac1{L_1\cdots L_r} \\
&= \frac1{L_1\cdots L_r} \sum_{\substack{ 0\le\ell_i,\ell_i'\le L_i \\ 1\le i\le r }}  \sum_{\substack{\bv y\in\Z_q^s \\ \bv {y\cdot a_i}\equiv w_i+\ell_i-\ell_i'}} 1 \\
&=\frac1{q^rL_1\cdots L_r} \sum_{c_1,\dots,c_r} \sum_{\substack{ 0\le\ell_i,\ell_i'\le L_i \\ 1\le i\le r }}  \sum_{\bv y\in\Z_q^s } \prod_{i=1}^re_q(c_i(w_i+\ell_i-\ell_i'-\bv{a_i\cdot y}))\\
&=\frac1{q^rL_1\cdots L_r} \sum_{c_1,\dots,c_r} e_q(c_1w_1+\cdots+c_rw_r) \sum_{\bv y\in\Z_q^s} e_q(-\bv y\bv\cdot (c_1\bv a_1+\cdots+c_r\bv a_r)) \prod_{i=1}^r \left|\sum_{0\le\ell_i\le L_i} e_q(\ell_ic_i)\right|^2 \\
&=\frac{q^{s-r}}{L_1\cdots L_r}\sum_{\substack{c_1,\dots,c_r \\ c_1\bv{a_1}+\cdots+c_r\bv{a_r}\equiv\bv 0\pmod q }}e_q(c_1w_1+\cdots+c_rw_r)\prod_{i=1}^r \left|\sum_{0\le\ell_i\le L_i} e_q(\ell_ic_i)\right|^2 \\
&\ll \frac{q^{s-r}}{L_1\cdots L_r} \sum_{\substack{c_1,\dots,c_r \\ c_1\bv{a_1}+\cdots+c_r\bv{a_r}\equiv\bv 0\pmod q }} \prod_{i=1}^r \min\left\{L_i^2,\|c_i/q\|^{-2}\right\}.
\end{align*}
For each $1\leq i\leq r$, we divide the possible values of $c_i$ into $\lceil q/H_i\rceil$ intervals of the form $[j_iH_i,(j_i+1)H_i)$ ($0\leq j_i<q/H_i$). Then
\[ N \ll \frac{q^{s-r}}{L_1\cdots L_r} \sum_{\substack{0\leq j_i<q/H_i\\ 1\leq i\leq r}} \sum_{\substack{j_iH_i\leq c_i<(j_i+1)H_i\\ c_1\bv a_1+\dots+c_r\bv a_r\equiv\bv 0\pmod q}} \prod_{i=1}^r \min\left\{L_i^2,\|c_i/q\|^{-2}\right\}. \]
By hypothesis, for fixed $j_1,\cdots,j_r$, there exists at most one tuple $(c_1,\cdots,c_r)$ with $j_iH_i\leq c_i<(j_i+1)H_i$ and $c_1\bv a_1+\cdots+c_r\bv a_r\equiv\bv 0\pmod q$. Hence the inner sum is summing over at most one tuple $(c_1,\cdots,c_r)$, and thus
\begin{align*} 
N &\ll \frac{q^{s-r}}{L_1\cdots L_r} \sum_{\substack{0\leq j_i<q/H_i\\ 1\leq i\leq r}}\prod_{i=1}^r\min\left\{L_i^2,\|j_iH_i/q\|^{-2}\right\}  \\
&\ll\frac{q^{s-r}}{L_1\cdots L_r} \prod_{i=1}^r\left(\sum_{0\le j_i< q/H_i}\min\left\{L_i^2,\frac{q^2}{j_i^2H_i^2}\right\}\right).
\end{align*}
We now treat the inner sum over $j_i$. There are two cases. If $H_iL_i\geq q$, then
\[ \sum_{0\leq j_i<q/H_i}\min\left\{L_i^2,\frac{q^2}{j_i^2H_i^2}\right\}=L_i^2+\sum_{1\leq j_i<q/H_i}\frac{q^2}{j_i^2H_i^2}\ll L_i^2+\frac{q^2}{H_i^2}\ll L_i^2. \]
In the other case when $H_iL_i\leq q$, we have
\[ \sum_{0\leq j_i<q/H_i}\min\left\{L_i^2,\frac{q^2}{j_i^2H_i^2}\right\}=\sum_{0\leq j_i<q/H_iL_i}L_i^2+\sum_{q/H_iL_i\leq j_i<q/H_i}\frac{q^2}{j_i^2H_i^2}\ll \frac{qL_i}{H_i}. \]
Combining the two cases together, we conclude that
\[ N\ll \frac{q^{s-r}}{L_1\cdots L_r} \prod_{i=1}^r\left(L_i^2+\frac{qL_i}{H_i}\right) =q^{s-r} \prod_{i=1}^r\left(L_i+\frac{q}{H_i}\right). \]
\end{proof}


\section{Upper Bound for the $\ell^1$ Norm of the Fourier Coefficients}\label{sec:l1norm}

In order to bound $S(f,r)$ using the classical Fourier analytic method, it is necessary to obtain a good upper bound for the $\ell^1$ norm of the Fourier coefficients of the charactersitic function of a generalized arithmetic progression of rank $r$. Let $A\subset\Z_q$ be an arbitrary subset. We will use the following normalization for the Fourier coefficients of $A$:
\[ \hat{A}(b)=\sum_{a\in A}e_q(ab), \]
\[ \|\hat{A}\|_1=\sum_{b\in\Z_q}|\hat{A}(b)|, \]
where $e_q(x)=\exp(2\pi ix/q)$. The main result of this section is the following.

\begin{thm}\label{thm:l1norm}
Let $A\subset\Z_q$ be a proper GAP of rank $r$ of the form
\[ A=\{a_0+a_1h_1+\ldots+a_rh_r\pmod q:0\leq h_i<H_i\}, \]
where $2\leq H_1,\cdots,H_r\leq q$. Then,
\[ \|\hat{A}\|_1\ll q(\log H_1)\cdots (\log H_r), \]
where the implied constant depends only on $r$.
\end{thm}

Instead of proving Theorem \ref{thm:l1norm} directly, we state and prove the corresponding result for $A\subset\Z_q^s$. We define the Fourier transform of $A$ in the usual manner. For $\bv b\in (b_1,\ldots,b_s)\in\Z_q^s$, define
\[ \hat{A}(\bv b)=\hat{A}(b_1,\ldots,b_s)=\sum_{\bv a\in A}e_q(\bv a\bv\cdot \bv b). \]
Define the $\ell^1$ norm of $\hat{A}$ to be
\[ \|\hat{A}\|_1=\sum_{\bv b\in\Z_q^s}|\hat{A}(\bv b)|. \]

\begin{thm}\label{thm:l1norm_gen}
Let $A\subset\Z_q^s$ be a proper GAP of rank $r$ of the form
\[ A=\{\bv a_0+h_1\bv a_1+\ldots+h_r\bv a_r\pmod q:0\leq h_i<H_i\}, \]
where $\bv a_0,\bv a_1,\ldots,\bv a_r\in\Z_q^s$ and $2\leq H_1,\ldots,H_r\leq q$. Then,
\[ \|\hat{A}\|_1\ll q^s(\log H_1)\cdots (\log H_r), \]
where the implied constant depends only on $r,s$.
\end{thm}

\begin{proof}
Since $A$ is proper, $H_i\leq q$. We have, by Lemma \ref{lem:simple_exp_sum},
\begin{align*} 
\|\hat{A}\|_1 &=\sum_{\bv b\in\Z_q^s}|\hat{A}(b)|=\sum_{\bv b\in\Z_q^s}\left|\sum_{h_1=0}^{H_1-1}e_q(h_1(\bv b\bv\cdot \bv a_1))\right|\cdots\left|\sum_{h_r=0}^{H_r-1}e_q(h_r(\bv b\bv\cdot \bv a_r))\right| \\
&\ll\sum_{\bv b\in\Z_q^s}\min\left\{H_1,\left\|\frac{\bv b\bv\cdot\bv a_1}{q}\right\|^{-1}\right\}\cdots \left\{H_r,\left\|\frac{\bv b\bv\cdot\bv a_r}{q}\right\|^{-1}\right\}.
\end{align*}
For each $1\leq i\leq r$, we divide $[0,q)$ into $H_i$ intervals of the form $[j_iq/H_i,(j_i+1)q/H_i)$ ($0\leq j_i<H_i$). We have
\[ \|\hat{A}\|_1\ll\sum_{\substack{0\leq j_i<H_i\\ 1\leq i\leq r}}\sum_{\substack{\bv b\in\Z_q^s\\ j_iq/H_i\leq\bv b\bv\cdot\bv a_i<(j_i+1)q/H_i}}\prod_{i=1}^r\min\left\{H_i,\left\|\frac{\bv b\bv\cdot\bv a_i}{q}\right\|^{-1}\right\}. \]
Since $A$ is proper, the equation $z_1\bv a_1+\cdots+z_r\bv a_r\equiv 0\pmod q$ has no nontrivial solutions $z_i\in\Z$ with $|z_i|<H_i$. For fixed $j_1,\cdots,j_r$, we apply Proposition \ref{prop:num_solns_gen} with $I_i=[j_iq/H_i,(j_i+1)q/H_i)$ ($1\leq i\leq r$) to conclude that the number of vectors $\bv b\in\Z_q^s$ with $\bv b\bv\cdot\bv a_i\in I_i$ is
\[ \ll q^{s-r}\prod_{i=1}^r\frac{q}{H_i}=\frac{q^s}{H_1\cdots H_r}. \]
Note also that for those vectors $\bv b\in\Z_q^s$ with $\bv{b\cdot a_i}\in I_i$, we have
\[ \left\|\frac{\bv{b\cdot a_i}}{q}\right\|^{-1}\ll\frac{H_i}{j_i}. \]
Hence
\[ \|\hat{A}\|_1\ll\sum_{\substack{0\leq j_i<H_i\\ 1\leq i\leq r}}\frac{q^s}{H_1\cdots H_r}\prod_{i=1}^r\min\left\{H_i,\frac{H_i}{j_i}\right\}=\frac{q^s}{H_1\cdots H_r}\prod_{i=1}^r\sum_{0\leq j_i<H_i}\min\left\{H_i,\frac{H_i}{j_i}\right\}. \]
The inner sum over $j_i$ above is easily evaluated to be
\[ H_i+\sum_{j_i=1}^{H_i-1}\frac{H_i}{j}\ll H_i\log H_i. \]
It then follows that
\[ \|\hat{A}\|_1\ll\frac{q^s}{H_1\cdots H_r}\prod_{i=1}^r(H_i\log H_i)=q^s(\log H_1)\cdots (\log H_r). \]
\end{proof}

Thus we have obtained upper bounds for $\ell^1$ norm of the Fourier coefficients of GAPs in $\Z_q$. By sending $q\rightarrow\infty$ in Theorem \ref{thm:l1norm_gen}, we immediately get an analogous result for GAPs in $\Z$. For $A\subset\Z$ and $x\in\R$, define
\[ \hat{A}(x)=\sum_{a\in A}e(ax), \]
where $e(y)=\exp(2\pi iy)$.

\begin{thm}\label{thm:l1norm_Z}
Let $A\subset\Z^s$ be a proper GAP of rank $r$ of the form
\[ A=\{\bv a_0+h_1\bv a_1+\ldots+h_r\bv a_r:0\leq h_i<H_i\}, \]
where $\bv a_0,\bv a_1,\ldots,\bv a_r\in\Z^s$. Then,
\[ \|\hat{A}\|_1=\int_0^1|\hat{A}(x)|dx\ll (\log H_1)\cdots (\log H_r), \]
where the implied constant depends only on $r,s$.
\end{thm}


\section{Bounds on Character Sums and Exponential Sums}\label{sec:gen_fr}

In this section, we complete the proof of Theorem \ref{main_thm}, Theorem \ref{thm:exp}, and Theorem \ref{main_gen_thm}, by the classical Fourier analytic techniques. This method is well known. See, for example, Chapter 23 in \cite{davenport2000}. 

\begin{lem}\label{lem:gen_fr}
Let $A\subset\Z_q^s$. Let $f:\Z_q^s\rightarrow\C$ be a function with $|f|\leq 1$. Then,
\[ \sum_{a\in A}f(\bv a)\leq\frac{\|\hat{f}\|_{\infty}\|\hat{A}\|_1}{q^s}. \]
\end{lem}

\begin{proof}
This is a simple consequence of the Plancherel Theorem. We abuse notation to write $A$ for the characteristic function of $A$. Then,
\[ \sum_{\bv a\in A}f(\bv a)=\sum_{\bv x\in\Z_q^s}f(\bv x)A(\bv x)=\frac{1}{q^s}\sum_{\bv b\in\Z_q^s}\hat{f}(\bv b)\hat{A}(\bv b)\leq \frac{\|\hat{f}\|_{\infty}\|\hat{A}\|_1}{q^s}. \]
\end{proof}

To prove Theorem \ref{main_thm} and \ref{thm:exp}, we simply need the $s=1$ version of the above lemma. When $s=1$, for a generic function $f$, one expects square root cancellation in $\hat{f}(b)$ and hence the bound $\|\hat{f}\|_{\infty}\ll\sqrt{q}$. This is indeed the case in our situations, when $f=\chi$ is a character, or when $f$ is the exponential of a polynomial.

\begin{proof}[Proof of Theorem \ref{main_thm}]
We use Lemma \ref{lem:gen_fr} with $f=\chi$ and Theorem \ref{thm:l1norm}. For any $b\in\Z_q$, we have
\[ \hat{\chi}(b)=\sum_{a\in\Z_q}\chi(a)e_p(ab)=\tau(\chi)\bar{\chi}(b). \]
For the last equality, see Chapter 9 of \cite{davenport2000}. The Gauss sum $\tau(\chi)$ satisfies $|\tau(\chi)|=\sqrt{p}$, and thus $\|\hat{\chi}\|_{\infty}=\sqrt{p}$. This completes the proof. 
\end{proof}

\begin{proof}[Proof of Theorem \ref{thm:exp}]
Suppose that $q$ is prime. Let $h(x)=c_dx^d+\ldots+c_1x+c_0\in\Z[x]$ be a polynomial of degree $2\leq d<q$ with $q\nmid c_d$. We use Lemma \ref{lem:gen_fr} with $f(x)=e_q(h(x))$ and Theorem \ref{thm:l1norm}. For any $b\in\Z_q$, we have
\[ \hat{f}(b)=\sum_ae_q(h(a)+ba)\ll d\sqrt{q} \]
by the classical Weil's bound. Hence $\|\hat{f}\|_{\infty}\ll d\sqrt{q}$. This completes the proof.
\end{proof}

Finally, we treat the multilinear character sum in Theorem \ref{main_gen_thm}.

\begin{proof}[Proof of Theorem \ref{main_gen_thm}]
We apply Lemma \ref{lem:gen_fr} with $f(a_1,\ldots,a_s)=\chi(a_1\cdots a_s)$ and use Theorem \ref{thm:l1norm_gen}. For any $\bv b=(b_1,\ldots,b_s)\in\Z_q^s$, we have
\begin{align*} 
\hat{f}(\bv b)&=\sum_{\bv a\in\Z_q^s}\chi(a_1\cdots a_s)e_q(\bv a\bv\cdot\bv b)=\left(\sum_{a_1\in\Z_q}\chi(a_1)e_q(a_1b_1)\right)\cdots\left(\sum_{a_s\in\Z_q}\chi(a_s)e_q(a_sb_s)\right) \\
&=\tau(\chi)^s\bar{\chi}(b_1)\cdots\bar{\chi}(b_s). 
\end{align*}
Hence, $\|\hat{f}\|_{\infty}=|\tau(\chi)|^s=q^{s/2}$. This completes the proof.
\end{proof}

It is natural to believe that Theorem \ref{main_thm} (and also Theorem \ref{main_gen_thm}) holds for all nontrivial characters $\chi$ (mod $q$), as is the case for the P\'{o}lya-Vinogradov inequality. For non-primitive characters $\chi$, one might argue as follows. Suppose that $\chi$ (mod $q$) is induced by the primitive character $\chi_1$ (mod $q_1$), and write $q=q_1r$ for some positive integer $r$. Then
\[ \sum_{n\in A}\chi(n)=\sum_{\substack{n\in A\\ (n,r)=1}}\chi_1(n)=\sum_{n\in A}\sum_{d\mid (n,r)}\mu(d)\chi_1(n)=\sum_{d\mid n}\mu(d)\sum_{n\in A_d}\chi_1(n), \]
where $A_d=\{n\in A:d\mid n\}$. In order to bound the inner sum over $A_d$, it is necessary to understand the structure of $A_d$. In the case when $A$ is an interval, it is obvious that $A_d$ is a dilate of an interval. However, when $A$ is a GAP of rank $r>1$, the structure of $A_d$ becomes mysterious. The question of bounding $S(\chi,r)$ for non-primitive characters $\chi$ still seems to be open.


\section{Open questions}\label{sec:conj}

Theorem \ref{thm:l1norm} shows that generalized arithmetic progressions have small $\ell^1$ norm of Fourier coefficients. The Littlewood-Gowers problem (the classical Littlewood problem in finite fields) asks what is the smallest possible $\|\hat{A}\|_1$ for all $A\subset\F_p$. The current best known result in this direction is by Sanders \cite{sanders2007}. See also \cite{green2006}. Using our notations from Section \ref{sec:l1norm}, we state their result as follows.

\begin{thm}\label{thm:littlewood}
Let $A\subset\F_p$. Then,
\[ \|\hat{A}\|_1\gg p(\log p)^{\frac{1}{2}}(\log\log p)^{-\frac{3}{2}}. \]
\end{thm}

The conjectural lower bound for $\|\hat{A}\|_1$ is $p\log p$ when $A$ is an interval. It is natural to seek for an inverse result for Theorem \ref{thm:littlewood}. More precisely, if we know that $\|\hat{A}\|_1$ is bounded by $p(\log p)^r$ for some $r$, is it true that $A$ is very close to a proper GAP of rank $r$? It is not clear at present what the correct formulation for this inverse result should be, so we will not make a rigorous conjecture here. This inverse problem was also suggested to the author by Ben Green.

Another open question is that whether the bound in Theorem \ref{main_thm} can be improved assuming GRH, as is done in the case when $A$ is an interval. Moreover, it would also be nice to be able to construct a family of characters $\chi$ (mod $q$) so that $S(\chi,r)$ is large, as is done by Paley \cite{paley1932} in the case $r=1$. See also \cite{goldmakher2011odd} and \cite{goldmakher2012even} for an improvement of Paley's result.

\end{document}